\documentclass[a4paper,11pt]{article}

\usepackage[top=3.0cm, bottom=3.0cm, inner=3.0cm, outer=3.0cm, includefoot]{geometry}

\usepackage{amssymb}
\usepackage{amsmath}
\usepackage{amsthm}
\usepackage{bbm}
\usepackage{relsize}
\usepackage{theoremref}
\usepackage{tikz}
\usepackage{subcaption}
\usepackage{caption}
\usepackage[figurename=Fig.]{caption}
\usepackage[labelfont=bf]{caption}
\DeclareCaptionLabelSeparator{none}{ }
\usepackage{color,soul}
\usepackage{verbatim}
\usepackage{diagbox}
\usepackage{url}
\usepackage{pgfplots}
\pgfplotsset{compat=1.18}

\usepackage[T1]{fontenc}
\usepackage[utf8]{inputenc}
\usepackage{authblk}

\title{A condition for non-negative Lin-Lu-Yau curvature\footnotetext{E-mail address: \texttt{moritz.hehl@uni-leipzig.de}}}

\author{Moritz Hehl}
\affil{Institute of Mathematics, Universit{\"a}t Leipzig, 04109 Leipzig, Germany }
\date{\today}

\theoremstyle{plain}
\newtheorem{lemma}{Lemma}[section]
\newtheorem{theorem}[lemma]{Theorem}

\theoremstyle{definition}
\newtheorem{definition}[lemma]{Definition}

\newtheorem{remark}[lemma]{Remark}

\newtheorem*{theoremMR}{Theorem~\ref{th:main_result}}
\newtheorem*{theoremS}{Theorem~\ref{th:sharpness}}

\numberwithin{equation}{section}

\begin{document}

\maketitle

\begin{abstract}
    We investigate the Ollivier-Ricci curvature and its modification introduced by Lin, Lu, and Yau on locally finite graphs. The main contribution of this work is a lower bound on the minimum vertex degree of a graph ensuring non-negative Lin-Lu-Yau curvature. Additionally, we examine the sharpness of this lower bound.
     \bigskip

    \bigskip


\end{abstract}

\section{Introduction and statement of results}

For over a century, Ricci curvature has been a fundamental concept in the study of smooth Riemannian manifolds, quantifying their deviation from being flat. Due to its significance, there has been considerable interest in generalizing Ricci curvature to more general metric spaces beyond the smooth setting. This pursuit has led, in recent decades, to the development of various curvature notions for non-smooth and discrete structures. Particularly, Forman introduced a Ricci curvature notion on cell complexes via the Bochner Weitzenb\"ock decomposition of the discrete Hodge Laplacian in 2003 \cite{Forman2003}. In 2009, Ollivier gave a notion of Ricci curvature for Markov chains on metric spaces via optimal transport \cite{Ollivier2009}. Ollivier's notion of Ricci curvature is particularly well suited for locally finite graphs. In this setting, a modification of Ollivier's notion of Ricci curvature was introduced in 2011 by Lin, Lu, and Yau \cite{LLY2011}. In 2012, Lin and Yau gave a discrete Ricci curvature notion for graphs via the Bakry-\'Emery calculus \cite{Lin2012}. Erbar and Maas defined a Ricci curvature based on the displacement convexity of the entropy in the same year \cite{EM2012}. Although these various notions reduce to the classical Ricci curvature in the setting of Riemannian manifolds, they are not equivalent for graphs.

Ollivier's notion of Ricci curvature can be seen as a discrete analogue to a lower Ricci curvature bound characterization established by Renesse and Sturm \cite{RS2005}. His approach leverages optimal transport theory, as his definition of curvature is based on the Wasserstein distance. In the context of locally finite graphs, Ollivier's notion of Ricci curvature has recently received considerable attention. In this setting, the \textit{Ollivier-Ricci curvature} $\kappa_{\alpha}$ is defined on the edges of a graph and depends on an \textit{idleness parameter} $\alpha \in [0,1]$. The modification introduced by Lin, Lu, and Yau is obtained by differentiating the curvature with respect to the idleness parameter. We will refer to this refinement as the \textit{Lin-Lu-Yau curvature} and denote it by $\kappa$.

Of particular interest is the case of non-negative Ricci curvature, for which several fundamental results are known in the setting of Riemannian manifolds, including the Liouville property, Buser inequality, Harnack inequality, Li-Yau inequality, and eigenvalue estimates. Extensive work has been carried out to develop discrete counterparts of these results, e.g., \cite{Muench2023,Jost2019, Chung2014, Erbar2018, Horn2017, Bauer2015, Lui2019, Muench2018}. 

The main contribution of this work is the establishment of a lower bound on the minimum degree of a graph that ensures non-negative Lin-Lu-Yau curvature.

\begin{theorem}\label{th:main_result}
    Let $G=(V,E)$ be a finite graph. If $G$ satisfies
    \[
        \delta(G) \geq \frac{2\vert V \vert}{3} -1,
    \]
    then $Ric(G) \geq 0$.
\end{theorem}

A similar result was previously established in \cite[Theorem 6.11]{MH2024} for the case of regular graphs. The proof relies on a simplified formula for the Lin-Lu-Yau curvature specific to regular graphs. In the present work, we extend this result by utilizing an intuitive formula for the Ollivier-Ricci curvature, introduced in \cite{Eidi2020}, which is also applicable to non-regular graphs.

After deriving this lower bound on the minimum degree of a graph that guarantees non-negative Lin-Lu-Yau curvature, we proceed to discuss the
sharpness of the bound.

\begin{theorem}\label{th:sharpness}
    Let $n \in \mathbb{N}$ such that $2n/3 \in 2 \mathbb{N}$ and $2n/3\geq 6$. Then, there exists a graph $G=(V,E)$ with $\vert V \vert = n$ and 
    \[
        \delta(G) = \frac{2\vert V \vert}{3} -2,
    \]
    that contains an edge $x \sim y$ with $\kappa(x,y) < 0$.
\end{theorem}

\section{Definitions and notations}\label{defs_and_notations}

We start by revisiting key concepts from Graph Theory and Optimal Transport Theory. Next, we present Ollivier's discrete Ricci curvature for graphs and its modification by Lin, Lu, and Yau.

\subsection{Graph Theory}

A \textit{simple graph} $G=(V,E)$ is an unweighted, undirected graph that contains no multiple edges or self-loops. The existence of an edge between two vertices $x,y \in V$ is denoted by $x \sim y$. For any two vertices $x, y \in V$, the \textit{combinatorial graph distance} is given by
\[
    d(x,y) = \inf\{n: \; x = x_0 \sim \ldots \sim x_n = y\}.
\]

The \textit{diameter} of $G$ is denoted by $diam(G) = \max_{x,y \in V} d(x,y)$. 

For $x \in V$ and $r \in \mathbb{N}$ we define the \textit{$r$-sphere centered at $x$} as $S_{r}(x) = \left\{y \in V: d(x,y) = r\right\}$ and the \textit{$r$-ball centered at $x$} as $B_{r}(x) = \left\{y \in V: d(x,y) \leq r\right\}$. Denote by $d_{x} = \vert S_{1}(x)\vert$ the \textit{degree} of a vertex $x\in V$. The \textit{minimum degree} of a graph $G$ is denoted by $\delta(G) = \min_{x\in V}d_{x}$. We call a graph $G$ \textit{locally finite} if every vertex has finite degree. The graph is said to be \textit{$d$-regular} if every vertex has the same degree $d$.

Given an edge $x \sim y$, we denote the set of common neighbors of $x$ and $y$ by $N_{xy} = S_{1}(x) \cap S_{1}(y)$.

\subsection{Ollivier-Ricci and Lin-Lu-Yau curvature}

The Wasserstein distance, a metric on the space of probability measures, plays a fundamental role in optimal transport theory.

\begin{definition}[Wasserstein distance]
    Let $G=(V,E)$ be a locally finite graph. Let $\mu_{1}, \mu_{2}$ be two probability measures on $V$. The \textit{Wasserstein distance} between $\mu_{1}$ and $\mu_{2}$ is defined as 
    \begin{equation} \label{eq:1}
        W_{1}(\mu_{1}, \mu_{2}) = \inf_{\pi \in \Pi(\mu_{1}, \mu_{2})} \sum_{x \in V} \sum_{y \in V} d(x,y)  \pi(x,y),
    \end{equation}
    where
    \begin{equation*}
        \Pi(\mu_{1}, \mu_{2}) = \left \{ \pi: V \times V \to [0,1]: \sum_{y \in V} \pi(x,y)=\mu_{1}(x), \; \sum_{x \in V} \pi(x,y)=\mu_{2}(y) \right \}. 
    \end{equation*}
\end{definition}

Imagine two distributions, $\mu_{1}$ and $\mu_{2}$, as piles of earth. A \textit{transport plan} $\pi \in \Pi(\mu_{1}, \mu_{2})$ transforms one pile of earth into another. The minimal effort which is required for such a transformation is measured by the Wasserstein distance. If the infimum in \ref{eq:1} is attained, we call $\pi$ an \textit{optimal transport plan} transporting $\mu_{1}$ to $\mu_{2}$. 

In optimal transport theory, it is well known that no mass needs to be transported when it is already shared between the two probability measures.

\begin{lemma}\thlabel{dontmove}
    Let $G=(V,E)$ be a locally finite graph. Let $\mu_{1}, \mu_{2}$ be two probability measures on $V$. Then there exists an optimal transport plan $\pi$ transporting $\mu_{1}$ to $\mu_{2}$ satisfying
    \begin{equation*}
        \pi(x,x) = \min\{\mu_{1}(x), \mu_{2}(x)\}
    \end{equation*}
    for all $x\in V$.
\end{lemma}

To each vertex $x\in V$ and idleness parameter $\alpha \in [0,1]$, we associate a probability measure $\mu_x^{\alpha}$ defined by

\begin{equation*}
    \mu_{x}^{\alpha}(y) =
    \begin{cases}
        \alpha, & \text{if $y = x$;}\\
        \frac{1-\alpha}{d_{x}}, & \text{if $y \sim x$;}\\
        0, & \text{otherwise.}
    \end{cases}
\end{equation*}

Then, the Ollivier-Ricci curvature is defined as follows.

\begin{definition}[Ollivier-Ricci curvature]
    Let $G=(V,E)$ be a locally finite graph. We define the \textit{$\alpha$-Ollivier-Ricci curvature} of an edge $x \sim y$ by
    \begin{equation*}
        \kappa_{\alpha}(x,y) = 1 - W_{1}(\mu_{x}^{\alpha}, \mu_{y}^{\alpha}).
    \end{equation*}
    The parameter $\alpha$ is called the \textit{idleness}.
\end{definition}

Following \cite{Eidi2020}, we present a more intuitive formulation of the Ollivier-Ricci curvature. To this end, denote by $\nu_{i}^{\alpha}$ the mass transported with distance $i$ under an optimal transport plan transporting $\mu_{x}^{\alpha}$ to $\mu_{y}^{\alpha}$. Then 
\begin{equation*}
    \sum_{i=0}^{3} \nu_{i}^{\alpha} = 1 \quad \text{and} \quad W_{1}(\mu_{x}^{\alpha}, \mu_{y}^{\alpha}) = \sum_{i=1}^{3} i \nu_{i}^{\alpha}.
\end{equation*}
Hence, we conclude 
\begin{equation*}
    \kappa_{\alpha}(x,y) = \nu_{0}^{\alpha} - \nu_{2}^{\alpha} - 2\nu_{3}^{\alpha}.
\end{equation*}
This formula plays a crucial role in proving our main result.

Recall that the $\alpha$-Ollivier-Ricci curvature depends on the idleness parameter. Ollivier initially considered the cases $\alpha = 0$ and $\alpha = \frac{1}{2}$, but its behavior across the full range of 
$\alpha$ remained an open question. Lin, Lu, and Yau \cite{LLY2011} first established that $\kappa_{\alpha}$ is concave in $\alpha \in [0,1]$. Using $\kappa_{1}(x,y) = 0$, one obtains that the function $h(\alpha) = \frac{\kappa_{\alpha}(x,y)}{1-\alpha}$  is increasing on the interval $[0,1)$. By further proving that $h(\alpha)$ is bounded, they concluded that the limit $\lim_{\alpha \to 1} h(\alpha)$ exists. Building on this result, Lin, Lu, and Yau introduced a modified version of the Ollivier-Ricci curvature, obtained by differentiating $\kappa_{\alpha}$ with respect to the idleness parameter $\alpha$.

\begin{definition}[Lin-Lu-Yau curvature]
    Let $G=(V,E)$ be a locally finite graph. The \textit{Lin-Lu-Yau curvature} of an edge $x \sim y$ is defined as 
    \begin{equation*}
        \kappa(x,y) = \lim_{\alpha \to 1} \frac{\kappa_{\alpha}(x,y)}{1-\alpha}.
    \end{equation*}
\end{definition}

In what follows, we write $Ric(G) \geq k$ if $\kappa(x,y) \geq k$ for all edges $x\sim y$ in $G$.

In \cite{Bourne2018}, Bourne et al. conducted a more detailed analysis of how $\kappa_{\alpha}$ depends on the idleness parameter $\alpha$, by introducing the Ollivier-Ricci idleness function $\alpha \mapsto \kappa_{\alpha}(x,y)$. They demonstrated that the idleness function is piecewise linear with at most three linear parts. Additionally, they determined the length of the final linear part.

\begin{theorem}[\cite{Bourne2018}, Theorem 4.4]
    Let $G=(V,E)$ be a locally finite graph and let $x,y \in V$ with $x \sim y$ and $d_{x} \geq d_{y}$. Then $\alpha \to \kappa_{\alpha}(x,y)$ is linear over $\left[\frac{1}{d_{x} + 1}, 1\right].$
\end{theorem}

As a direct consequence of the mean value theorem, they derived the following relation between the $\alpha$-Ollivier-Ricci curvature and its modification by Lin, Lu, and Yau.

\begin{theorem}\label{th:relation_curvature}
    Let $G=(V,E)$ be a locally finite graph and let $x,y \in V$ with $x \sim y$ and $d_{x} \geq d_{y}$. Then
    \begin{equation*}
        \kappa_{\alpha}(x,y) = (1-\alpha)\kappa(x,y)
    \end{equation*}
    for $\alpha \in \left[\frac{1}{d_{x} +1},1\right]$.
\end{theorem}

\begin{remark}
   Loisel and Romon had previously established this relation in \cite{Loisel2014} for $\alpha \geq \frac{1}{2}$ in 2014.
\end{remark}

Consequently, for large values of $\alpha$, the Lin-Lu-Yau curvature aligns with the $\alpha$-Ollivier-Ricci curvature up to a scaling factor.

In the case of regular graphs, the optimal transport problem can be reduced to an optimal assignment problem between subsets of the spheres. This observation was leveraged in \cite{MH2024} to derive a more concise expression for the Lin-Lu-Yau curvature.

\begin{theorem}[\cite{MH2024}, Theorem 4.3]\label{th:Lin_Lu_Yau_curvature}
    Let $G=(V,E)$ be a locally finite graph. Let $x,y \in V$ be of equal degree $d$ with $x \sim y$. Then the Lin-Lu-Yau curvature
    \begin{equation*}
        \kappa(x,y) = \frac{1}{d}\Biggl(d+1 - \inf_{\phi} \mathlarger{\sum}_{z \in S_{1}(x) \setminus B_{1}(y)}d(z, \phi(z)) \Biggr),
    \end{equation*}
    where the infimum is taken over all bijections $\phi:S_1(x)\setminus B_1(y) \to S_1(y)\setminus B_1(x)$.
\end{theorem}

\section{Condition for non-negative Lin-Lu-Yau curvature}

In this section, we present the main result of this work: A condition on the minimum degree of a graph $G$, ensuring that it satisfies $Ric(G) \geq 0$. Additionally, we discuss the sharpness of this lower bound.

\begin{theoremMR}
    Let $G=(V,E)$ be a finite graph. If $G$ satisfies
    \[
        \delta(G) \geq \frac{2\vert V \vert}{3} -1,
    \]
    then $Ric(G) \geq 0$.
\end{theoremMR}

For the proof of Theorem~\ref{th:main_result}, we require the following lemma, which guarantees that the diameter of a graph is at most two when the minimal degree is sufficiently large.

\begin{lemma}\thlabel{lem:diam_bound}
    Let $G=(V,E)$ be a finite graph with minimum degree satisfying
    \begin{equation*}
        \delta(G) \geq \frac{\vert V \vert - 1}{2}.
    \end{equation*}
    Then the diameter of $G$ satisfies $diam(G) \leq 2$.
\end{lemma}

\begin{proof}
    Assume that $diam(G) > 2$. Hence, there exist vertices $x,y \in V$ such that $d(x,y) >2$, and therefore $S_{1}(x) \cap S_{1}(y) = \emptyset$. As $x,y \notin S_{1}(x) \cup S_{1}(y)$, we have 
    \begin{equation*}
        \vert S_{1}(x) \cup S_{1}(y) \vert \leq \vert V \vert - 2.
    \end{equation*}
    This contradicts 
    \begin{equation*}
        \vert S_{1}(x) \cup S_{1}(y)\vert = d_x + d_y\geq 2\delta(G) \geq \vert V \vert - 1.
    \end{equation*}
    Therefore, our assumption was incorrect, and it follows that $diam(G) \leq 2$ must hold.
\end{proof}

We are now prepared to prove Theorem~\ref{th:main_result}.

\begin{proof}[Proof of Theorem~\ref{th:main_result}]
    Let $x \sim y$ be an arbitrary edge in $G$ and assume, without loss of generality, that $d_x \leq d_y$. Choose
    \[
        \alpha = \frac{1}{d_x + 1}.
    \]
    According to Theorem~\ref{th:relation_curvature}, the equality
    \[
        \kappa(x,y) = \frac{1}{1-\alpha}\kappa_{\alpha}(x,y)
    \]
    holds. Hence, it suffices to prove that $\kappa_{\alpha}(x,y) \geq 0$. Let $\pi$ be an optimal transport plan, transporting $\mu_{x}^{\alpha}$ to $\mu_y^{\alpha}$. Denote by $\nu_{i}^{\alpha}$ the mass transported with distance $i$ under $\pi$. According to Lemma~\ref{lem:diam_bound}, we have $\nu_3^{\alpha} = 0$ and therefore
    \[
        \nu_0^{\alpha} + \nu_{1}^{\alpha} + \nu_2^{\alpha} = 1.
    \] 
    Thus, we obtain
    \begin{align*}
        \kappa_{\alpha}(x,y) &= 1 - W_1(\mu_x^{\alpha}, \mu_y^{\alpha})\\
        &= \nu_0^{\alpha} + \nu_{1}^{\alpha} + \nu_2^{\alpha} - \nu_{1}^{\alpha} - 2\nu_2^{\alpha}\\
        &= \nu_0^{\alpha} - \nu_2^{\alpha}.
    \end{align*}
    
    According to \thref{dontmove}, we may assume
    \[
        \pi(z,z) = \min\{\mu_{x}^{\alpha}(z),\mu_y^{\alpha}(z)\} 
    \]
    for every $z\in V$. Therefore, $\pi(z,z) = (1-\alpha)/d_y$ for every $z \in N_{xy}$, $\pi(x,x) = (1-\alpha)/d_y$, and $\pi(y,y) = (1-\alpha)/d_x$. This yields
    \[
        \nu_{0}^{\alpha} = \frac{1-\alpha}{d_y}(\vert N_{xy}\vert +1) + \frac{1-\alpha}{d_x}.
    \]
    Since $\nu_2^{\alpha} \leq 1 - \nu_0^{\alpha}$, we conclude
    \begin{align*}
        \kappa_{\alpha}(x,y) &\geq 2\nu_0^{\alpha} - 1\\
        &= 2(1-\alpha)\left[\frac{\vert N_{xy}\vert +1}{d_y} + \frac{1}{d_x} - \frac{1}{2(1-\alpha)}\right] \\
        &= 2(1-\alpha)\left[\frac{\vert N_{xy}\vert +1}{d_y} + \frac{1}{d_x} - \frac{d_x+1}{2d_x}\right] \\
        &= 2(1-\alpha)\left[\frac{\vert N_{xy}\vert +1}{d_y} + \frac{1}{2d_x} - \frac{1}{2}\right] \\
        &\geq 2(1-\alpha)\left[\frac{2\vert N_{xy}\vert +2}{2d_y} + \frac{1}{2d_y} - \frac{1}{2}\right]\\
        &\geq (1-\alpha)\left[\frac{2\vert N_{xy}\vert +3}{d_y} - 1\right].
    \end{align*}
    Note that
    \[
        \vert V \vert \geq \vert S_1(x) \cup S_1(y) \vert 
        = d_x + d_y - \vert N_{xy} \vert 
        \geq 2\delta(G) - \vert N_{xy}\vert \geq \frac{4\vert V \vert}{3} - 2 - \vert N_{xy} \vert.
    \]
    Hence, we conclude 
    \[
        \vert N_{xy} \vert \geq \frac{\vert V \vert}{3} -2.
    \]
    Furthermore, we have
    \[
        d_y \leq \vert V \vert - d_x + \vert N_{xy} \vert \leq \vert V \vert - \delta(G) + \vert N_{xy} \vert \leq \frac{\vert V \vert}{3} + 1 + \vert N_{xy} \vert.
    \]
    Putting the inequalities together yields 
    \begin{align*}
        2\vert N_{xy} \vert + 3 - d_y &\geq 2\vert N_{xy} \vert + 3 - \left[ \frac{\vert V \vert}{3} + 1 + \vert N_{xy} \vert \right]\\
        &= \vert N_{xy} \vert + 2 - \frac{\vert V \vert}{3} \\
        &\geq 0,
    \end{align*}
        
     and therefore
    \[
        \kappa_{\alpha}(x,y) \geq (1-\alpha)\left[\frac{2\vert N_{xy}\vert +3}{d_y} - 1\right] \geq 0.
    \]
    Since the edge $x \sim y$ was chosen arbitrarily, we conclude
    \[
        Ric(G) \geq 0.
    \]
\end{proof}

Next, we discuss the sharpness of the lower bound.

\begin{theoremS}
    Let $n \in \mathbb{N}$ such that $2n/3 \in 2 \mathbb{N}$ and $2n/3\geq 6$. Then, there exists a graph $G=(V,E)$ with $\vert V \vert = n$ and 
    \[
        \delta(G) = \frac{2\vert V \vert}{3} -2,
    \]
    that contains an edge $x \sim y$ with $\kappa(x,y) < 0$.
\end{theoremS}

\begin{proof}
    Define the vertex set
    \begin{equation*}
        V = \{x,y, z_{0}, \dots, z_{l-3}, x_{0}, \dots, x_{l}, y_{0}, \dots, y_{l}, v\},
    \end{equation*}
    where $l \geq 2$. Next, we add the edges
    \begin{itemize}
        \item $x\sim y$,
        \item $x \sim z_{i}$ for $i = 0,\dots, l-3$,
        \item $x \sim x_{i}$ for $i = 0,\dots, l$
        \item $y \sim z_{i}$ for $i = 0,\dots, l-3$,
        \item $y \sim y_{i}$ for $i = 0,\dots, l$.
    \end{itemize}
    to the set of edges $E$. Note that there are no vertices $z_i$ in the case where $l=2$. The vertices $x$ and $y$ are of degree $2l$, and
    \begin{align*}
        N_{xy} &= \{z_{0}, \dots, z_{l-3}\}, \\
        S_1(x)\setminus B_1(y) &= \{x_{0}, \dots, x_{l}\}, \\S_1(y)\setminus B_1(x)&= \{y_{0}, \dots, y_{l}\}.
    \end{align*}
    Next, we add for every $z_{i} \in N_{xy}$ the edges $z_{i} \sim x_{j}$ and $z_{i} \sim y_{j}$ for $j = 0, \dots, l$.
    Thus, every $z_{j} \in  N_{xy}$ satisfies
    \[
        d_{z_j} = 2(l +1) + 2 > 2l.
    \]

    Additionally, add the edges $x_{i} \sim x_{j}$ and $y_i \sim y_j$ for $i \neq j$, as well as $v \sim x_i$ and $v \sim y_i$ for $i = 0, \dots,l$. Hence, we obtain
    \begin{align*}
        d_{x_i} &= 1 + l + l-2 + 1 = 2l \quad \forall i \in \{0, \dots, l\},\\
        d_{y_i} &= 1 + l + l-2 + 1 = 2l \quad \forall i \in \{0, \dots, l\}, \\
        d_v &= 2(l+1) > 2l.
    \end{align*}
    Therefore, the resulting graph satisfies $\delta(G) = 2l$. Furthermore, we have $\vert V \vert = 3l + 3$, and therefore
    \[
        \delta(G) = \frac{2\vert V \vert}{3} - 2.
    \]
    
    Since $d_x = d_y$, we can apply Theorem~\ref{th:Lin_Lu_Yau_curvature}. Note that, by construction, the equality $d(x_i,y_j) = 2$ holds for every $i,j \in \{0,\dots,l\}$. Hence, $d(z, \phi(z)) = 2$ for every $z\in S_{1}(x) \setminus B_{1}(y)$ and every bijection $\phi: S_{1}(x) \setminus B_{1}(y) \to S_{1}(y) \setminus B_{1}(x)$. Thus, we conclude
    \begin{align*}
        \kappa(x,y) &= \frac{1}{d_x}\Biggl(d_x + 1 - \inf_{\phi} \mathlarger{\sum}_{z \in S_{1}(x) \setminus B_{1}(y)}d(z, \phi(z))\Biggr)\\
        &= \frac{1}{d} \Biggl(d_x + 1 - 2 \vert S_{1}(x) \setminus B_{1}(y)\vert\Biggr) \\
        &= \frac{1}{d} \Biggl(d_x + 1 - 2\left(\frac{d_x}{2} +1\right)\Biggr) \\
        &= -\frac{1}{d}.
    \end{align*}
\end{proof}

\bigskip

{\bf{Acknowledgements:}} I would like to express my deepest gratitude to Prof. Dr. Renesse and Dr. M{\"u}nch for their invaluable support and insightful feedback throughout the course of this work.

\section*{Declarations}

{\bf{Data availability:}} No datasets were generated or analysed during the current study.

{\bf{Competing interests:}} The author certifies that he has no affiliations with or involvement in any organization or entity with any financial interest or non-financial interest in the subject matter or materials discussed in this manuscript.

{\bf{Funding:}} Partial financial support was received from the BMBF (Federal Ministry of Education and Research) in DAAD project 57616814 (SECAI, School of Embedded Composite AI).

\end{document}